\documentclass{article}
\usepackage[paperwidth=7in, paperheight=10in, margin=.875in]{geometry}
 \usepackage[backref,colorlinks,linkcolor=red,anchorcolor=green,citecolor=blue]{hyperref}
\usepackage{amsfonts,amssymb}
\usepackage{amsmath,amsthm}
\usepackage{graphicx}
\usepackage{cite}
\usepackage{enumerate}
\newcommand{\bigzero}{\mbox{\normalfont\Large\bfseries 0}}

\newcommand{\bigidentity}{\mbox{\normalfont\Large\bfseries I}}

  \newtheorem{theorem}{Theorem}
  \newtheorem{lemma}[theorem]{Lemma}
   \allowdisplaybreaks
\begin{document}
 \title{Hypocoercivity of the linearized BGK equation with stochastic coefficients %\thanks{Received date, and accepted date (The correct dates will be entered by the editor).}
 \date{ }
 }

          %For each author, make a block with the following macros:

          \author{T. Herzing \footnote{University of Bamberg, Kärntenstraße 7, 96052 Bamberg, (tobias.herzing@uni-bamberg.de)}
          \and C. Klingenberg\footnote{University of Würzburg, Emil-Fischer-Str. 40, 97074 Würzburg, (klingen@mathematik.uni-wuerzburg.de).} 
          \and M. Pirner \footnote{University of Würzburg, Emil-Fischer-Str. 40, 97074 Würzburg, (marlies.pirner@mathematik.uni-wuerzburg.de}
          }

         \pagestyle{myheadings} \markboth{Hypocoercivity of the linearized BGK equation with stochastic coefficients}{T. Herzing, C. Klingenberg, M. Pirner} \maketitle

{\bf abstract:}
We consider an approximation of the Boltzmann equation, the Bathnagar-Gross-Krook (BGK) equation. This equation is used in many applications because it is very efficient in numerical simulations. In this paper we study the effect of randomness on a BGK-model. We prove exponential decay rate to a global equilibrium. % This decay rate can be proven to be independent of the stochastic influence in a physically reasonable norm. 
In addition we prove the decay rate of the n-th derivative with respect to the stochastic variable of the solutions. % Our strategy is based on Lyapunov's method as it is presented in
% \cite{Achleitner2016,Achleitner2018}. The matrices we need for a Lyapunov estimate now depend on the stochastic variable. This requires a careful analysis of the random effect.
The novelties are i.) for the first time hypocoercivity is shown for a linearized BGK model that conserves mass, momentum and energy with randomness in the collision frequency, ii.)new estimates for the decay of the derivatives of the solution with respect to the stochastic variable, which is very useful in applications.
   %       \end{abstract}
   
{\bf keywords:}  linear BGK-equation with uncertainties; hypocoercivity;  decay estimate; Lyapunov's direct method

% \begin{AMS} 35A24; 35B30; 35Q20; 82B40
%\end{AMS}
    \section{Introduction}\label{intro}

In this paper, our aim is to study the decay to equilibrium of the solution of a linearized BGK model introduced in \cite{Achleitner2018} with a random parameter in the collision frequency. 
We begin by introducing the non-linear BGK model  \begin{align}
    \partial_t f + v \cdot \nabla_x f = \sigma (M-f).
    \label{BGK}
\end{align}
Here $f(x,v,t)$ is the number density distribution of one species of gas  with respect to the phase space measure $dx dv$. \textcolor{black}{Here $x \in (\frac{L}{2 \pi} \mathbb{T})^d$ in the $d-$dimensional torus} of side length $L$ is the position of the coordinate in phase space. $v \in \mathbb{R}^d$ is the velocity coordinate in dimension $d \in \mathbb{N}$ and $t \geq 0$ is the time. In the following, we will consider the dimension $d=1$ . The relaxation operator on the right-hand side of \eqref{BGK} involves the Maxwellian
\begin{align*}
 \textcolor{black}{   M = \frac{n}{\sqrt{2 \pi T}} \exp \left( - \frac{|v-u|^2}{2 T}\right)}
\end{align*}
depending on the macroscopic quantities (density $n$, mean velocity $u$, temperature $T$) defined as
\begin{align*}
    \int f(v) \begin{pmatrix} 1 \\ v \\  (v-u)^2 \end{pmatrix} dv = \begin{pmatrix} n \\ n u \\  n T \end{pmatrix}.
\end{align*}
Moreover, the BGK model \eqref{BGK} contains the collision frequency $\sigma$. The purpose of the collision operator in \eqref{BGK} is to provide an approximation of the %multi-species 
Boltzmann collision operator that is more computationally tractable, but still maintains important structural properties. It was first introduced in \cite{BGK1954} by Bathnagar, Gross and Krook. It has the same collision invariants as the Boltzmann operator (which lead to conservation of number of particles, momentum and energy) and it satisfies an H-Theorem. 

One natural aspect of kinetic equations are uncertainties. The form of some terms (for instance of the collision frequency) in the equations are usually unjustified due to  modelling errors. The blurred measurements are typically not enough to sufficiently determine all coefficients. %Often such quantities contain uncertainties coming from modeling errors. 
Therefore, in this paper, we consider the collision frequency $\sigma(z)$ depending on a random parameter $z$. In the whole paper, we assume that this dependency is continuous. 

Now, the aim of the paper is to study the regularity and the large-time behavior of $f$ and of the derivatives $\partial_z^{(n)}f$ in dimension $d=1$. This is based on the hypocoercivity theory which has been studied 
%This article is concerned with hypocorecivity-estimates of a randomized BGK-model\footnote{Named after the physicists Bhatnagar-Gross-Krook\cite{BGK1954}} in one dimension. Hypocorecivity was made widely known by Villani \cite{Villani2009} for equations of the form $ \frac{\partial}{\partial t} f = - Lf$, where the generator $L$ is not coercive, but where solutions still exhibit exponential decay in time. The long-time behavior has been studied 
for a large variety of equations. Some considerable examples in the deterministic case are the Fokker-Plank equations \cite{Achleitner2015,Arnold2014}, linear kinetic equations \cite{Dolbeault2015,favre2020hypocoercivity,neumann2015kinetic,Bouin,Desvillettes,Dolbeault,Herau}, a multi-species Boltzmann system \cite{daus2016hypocoercivity} as well as the BGK-equations \cite{Achleitner2016,Achleitner2018,liu2019hypocoercivity}. Especially in \cite{Achleitner2016,Achleitner2018,Arnold2014} it was an issue to find sharp exponential decay rates. 
In the random case, this has been extended in many cases for example to linear kinetic equations
%Uncertainty is natural for many physical equations. This may have various reasons, like modelling errors or blurred measurements. Thus it is not always sufficient to look for the exact solution. Also a careful study of the uncertainty effect and their long-time behavior is required. Such an analysis was made for linear equations 
in \cite{Liu2018,Li2017,arnold2020sharp}, for the multi-species Boltzmann equation in \cite{daus2019multi}, the Vlasov-Poisson-Fokker-Planck system \cite{Vlasov} and equations used for traffic modelling \cite{vehicular}. \textcolor{black}{Such a study of the regularity and the large-time behavior of $f$ and of the derivatives $\partial_z^{(n)}f$ allows to adopt the gPC framework for its possible fast convergence. To do that, one mainly needs to prove that the perturbation in the solution continuously depends on the perturbation  where one chooses to perform linearization. According to the standard spectral method theory, the higher degree of continuity means the faster convergence.} \textcolor{black}{For example such a study is provided by \cite{Liu2018} for the Boltzmann equation.}

In this paper, we want to understand the regularity and decay to equilibrium of the function $f$ and also of its derivatives $\partial_z^{(n)}f$ for all $n \in \mathbb{N}$.
We denote by $d\tilde{x}:= L^{-d} dx$ the normalized Lebesque measure and consider normalized initial data
\begin{align}
    \int \int f^{I} d\tilde{x} dv =1, \quad \int \int v f^{I} d\tilde{x} dv = 0, \quad \int \int v^2  f^{I} d\tilde{x} dv = 1
    \label{normalized}
\end{align}
Now, we linearize the BGK equation \eqref{BGK} around the unique space-homogeneous steady state
\begin{align*}
    \mathbb{M}_1(v) = \frac{1}{(2 \pi )^{1/2}} \exp \left( - \frac{v^2}{2} \right)
\end{align*}
as it is performed in \cite{Achleitner2018}. For this, we consider the splitting $f(x,v,t)= \mathbb{M}_1(v) + h(x,v,t,z)$ with the macroscopic quantities of $h$ defined as
\begin{align}
\begin{split}
\label{estunate35}
\omega(x,t,z) &:= \int\limits_{\mathbb{R}} h(x,v,t,z) \; \mathrm{d} v, \quad
\mu(x,t,z) := \int\limits_{\mathbb{R}} v h(x,v,t,z) \; \mathrm{d} v  \\
\tau(x,t,z) &:= \int\limits_{\mathbb{R}} v^2 h(x,v,t,z) \; \mathrm{d} v 
\end{split}
\end{align}
If we insert this ansatz into \eqref{BGK}, do a Taylor expansion of $M$ with respect to $\omega, \mu, \tau$ around $0$, and take only the linear terms, one can derive similar as it is done in \cite{Achleitner2018} the linearized equation
%More exactly, we consider the equation
\begin{align} \label{BGKEquation}
\partial_t h(x,v,t,z) + v \partial_x h(x,v,t,z) =\sigma(z) \; \mathcal{L} \left( h (x,v,t,z) \right) %\mathcal{L}_z \left( h (x,v,t,z) \right)  
\end{align} 
%with $\mathcal{L}_z$ defined as
%\begin{align*}
%\mathcal{L}_z \left( h (x,v,t,z) \right) := .
%\end{align*}
%Here $ \sigma(z) $ is a continuous function from $ \mathbb{O} \subseteq \mathbb{R} \rightarrow \mathbb{R} $ and $ \mathcal{L} $ is given as
with
\begin{align*}
\mathcal{L}(h) := \mathbb{M}_1(v) \left[ \left( \textcolor{black}{\frac{3}{2} }- \frac{v^2}{2} \right) \omega(h) + v \mu(h) + \left( - \frac{1}{2} + \frac{v^2}{2} \right) \tau(h) \right] - h,
\end{align*} 
%\begin{align} \label{estimate32}
%&\int\limits_{\tilde{T}} \omega(x,0,z) \; \mathrm{d} x = 0, &
%&\int\limits_{\tilde{T}} \mu(x,0,z) \; \mathrm{d} x = 0, &
%\int\limits_{\tilde{T}} \tau(x,0,z) \; \mathrm{d} x = 0,
%\end{align}
%\textcolor{cyan}{
Uncertainties are also important from the point of view of numerics. Nowadays many numerical methods with the aim to address the issues related to uncertainties have been developed. Well known numerical methods are the Monte-Carlo method, the moment equation approach and the perturbation methods.
In addition there are spectral-methods like the (Galerkin) generalized polynomial chaos method and the stochastic collocation method.
A review of spectral type methods can be found in \cite{Xiu2010}. %The stochastic collocation method is a non-intrusive method, which means that it takes recourse to a deterministic solver. %The main steps are
%\begin{enumerate}
%\item choose a set of nodes in the random space
%\item solve the deterministic problem at each node
%\item construct (polynomial chaos) polynomials that coincide with the solution at each node.
%\end{enumerate}
%By contrast to this proceeding the generalized polynomial chaos method needs an implementation independent of the deterministic solver, which is called intrusive. Here, roughly speaking, one first tries to get rid of the stochastic dimension by some orthogonal polynomial expansions and then one needs to solve the resulting system of equations.
One thing spectral methods have in common is that they provide a higher order of accuracy if the solution has a high level of regularity. Thus it is a common procedure to check the derivatives or show boundedness or even decay in time in some reasonable norm. In this context we point out the paper by Li and Wang \cite{Li2017}, where such a regularity condition has been studied for a large set of kinetic equations. Their paper contains the linear BGK-operator with constant velocity and temperature (where only mass is conserved).
%}

The aim of this article is to extend the results in \cite{Li2017} to the linearized BGK equation \eqref{BGKEquation}. In our case we have a dependency on the macroscopic quantities $\omega,\textcolor{black}{ \mu}, \tau$ instead of a constant Maxwell distribution with a fixed constant mean velocity and temperature. In our case, not only the mass is conserved but also momentum and energy. We will show exponential decay in time with a rate $-\lambda$  independent of the random variable and $\lambda$ strictly positive in a physical reasonable norm. To do so, we use the technique developed in \cite{Achleitner2016, Achleitner2018}. The advantage of this approach is that we directly inherit the optimization strategies made in these articles. 
%
%This means we are able to find sharp decay rates. % which hold in a p-almost sure sense. 
%More concrete, our decay rate holds true for every possible realization of the random variable $z$. Further $z$ might realize in such a way, that the decay rate found is indeed sharp in the sense of \cite{Achleitner2016,Achleitner2018}. Such the decay rate can be seen as an a priory estimate of a sharp lower bound. This is, bevor z realizes, no better estimates for the decay rate can be made. 
\textcolor{black}{In addition to the aforementioned differences in our model, we also differ from \cite{Li2017} that we look for sharp decay rates. To achieve this, we adapt a method proposed by \cite{Achleitner2016} for the deterministic case.
However, in contrast to the literature for sharp decay rates \cite{Achleitner2016} we need to find estimate which hold for every possible realization of $z$. This requires careful modifications of the already known approaches.}

This has to be understood as kind of an a priori estimate, which means that we find sharp decay rates which serve as lower bound for all possible realizations. This means, the slowest possible decay rate which can be realized tends to be sharp in the sense of \cite{Achleitner2016, Achleitner2018}. %\footnote{In these articles the authors present a extension of Lyapunov's direct method to infinite dimensions. In contrast to the finite dimensions counterpart, which leads to exact decay estimates, the presented approach does not yet fulfill the same goal. Nevertheless this method reveals good and reasonable approximations.}.
Furthermore, the resulting decay rates are directly computable. Moreover, we show that this decay rate $\lambda$ also holds for the decay of the derivatives in the random space. That means, computing such a decay rate $\lambda$ for the underlying BGK equation once, gives us immediately a decay rate for the derivatives in the random space.

In summary, the novelty of this article consists of showing hypocoercivity for the linearized BGK model with randomness in the collision frequency conserving mass, momentum and energy. We include new estimates for the decay of the derivatives of the solution with respect to the stochastic variable, which is very useful in applications.

In section \ref{ChapterLinModel}, we will begin by writing the linearized BGK-model with uncertainties in one space dimension as an infinite system of ODEs similar as it is done in \cite{Achleitner2018}.
Section \ref{ChapterEstimate} is divided in three parts. In the first subsection we will extend Lyapunov's direct method in infinite dimensions to equations with a random parameter in the collision frequency. This is a crucial step on our search for decay rates and directly leads to our first decay estimate presented in the second part of this section. Finally, in the third part we deal with decay estimates in $z$-derivatives. The main idea here is to benefit from two Gronwall-like estimate theorems presented in \cite{Li2017}.% As a consequence, we are not longer able to speak of sharp decay estimates in the derivatives, only of lower bounds.

%\section{A linearized BGK model with uncertainties
%We want to extend the linearized BGK equation established in \cite{Achleitner2018}. Therefore we will add a random component in both, the initial data and the RHS operator. We expect $z$ to be a continuous random variable $\left( \Omega, \Sigma, P \right) \mapsto \left( \mathbb{R}, B \right)$\footnote{Here $\Sigma$ is an $\sigma$-algebra on $\Omega$ and B denotes the Borel-$\sigma$-algebra.}, which maps from a random space to $\mathbb{O} \subseteq \mathbb{R}$, where $\mathbb{O}$ is either $\mathbb{R}$ or an interval. 
%where we set

%Because of the conservation of mass, momentum and energy (see \cite{Achleitner2018}) we have

%where $\tilde{T} := \frac{L}{2 \pi} T$ is the torus of side length $L$. To prepare for the following proofs
\section{Transformation of the linearized BGK equation to an infinite system of ODEs}
\label{ChapterLinModel}
To prepare for the following proofs, we want to rewrite \eqref{BGKEquation} into an (infinite dimensional) system of differential equations as it is done in the deterministic case in \cite{Achleitner2018,Achleitner2016}.
We do this by expanding $h(x,v,t,z)$ in a Fourier series in $x$
\begin{align*}
h(x,v,t,z) = \sum_{k \in \mathbb{Z}} h_k(v,t,z) \mathrm{e}^{\mathrm{i} k \frac{2 \pi}{L} x}.
\end{align*}

Then, we will expand $h_k(\cdot,t,z) \in L^2\left( \mathbb{R} ; \mathbb{M}_1^{-1}(v) \right)$ in normalized Hermite functions
\begin{align*}
    g_m(v):= (\pi m!)^{-1/2} H_m(v) \exp \left(- \frac{v^2}{2}\right), \quad H_m(v):= (-1)^m \exp \left(\frac{v^2}{2}\right) \frac{d^m}{dv^m} \exp \left(- \frac{v^2}{2}\right)
\end{align*}
by writing
\begin{align*}
h_k(v,t,z) = \sum_{m=0}^{\infty} \hat{h}_{k,m}(t,z) g_m(v) \quad \text{with} \quad
\hat{h}_{k,m}(t,z) = \langle h_k(v,\cdot,\cdot), \; g_m(v) \rangle_{L^2(\mathbb{M}_1^{-1})}.
\end{align*}
For each $k \in \mathbb{Z}$ the vector $\hat{h}_k(t,z) = \left( \hat{h}_{k,0}(t,z), \; \hat{h}_{k,1}(t,z), \; \dots \right)^T \in \ell^2(\mathbb{N}_0)$ contains all Hermite coefficients of $h_k(\cdot,t,z)$.
Note that the first three normalized Hermite functions are given by
\begin{align*}
    g_0(v)= \mathbb{M}_1(v), ~ g_1(v)=v \mathbb{M}_1(v), ~ g_2(v)= \frac{v^2-1}{\sqrt{2}} \mathbb{M}_1(v).
\end{align*}
Moreover, we have

\begin{align}
\hat{h}_{k,0}(t,z) &= \int_{\mathbb{R}} h_k(v,\cdot,\cdot) g_0(v) \mathbb{M}_1^{-1}(v) \; \mathrm{d} v = \omega_k(t,z) \label{estimate26} \\
\hat{h}_{k,1}(t,z) &= \int_{\mathbb{R}} h_k(v,\cdot,\cdot) g_1(v) \mathbb{M}_1^{-1}(v) \; \mathrm{d} v = \mu_k(t,z) \label{estimate27}\\
\hat{h}_{k,2}(t,z) &= \int_{\mathbb{R}} h_k(v,\cdot,\cdot) g_2(v) \mathbb{M}_1^{-1}(v) \; \mathrm{d} v = \frac{1}{\sqrt{2}} \left( \tau_k(t,z) - \omega_k(t,z) \right). \label{estimate28}
\end{align}
where $\omega_k, \mu_k, \tau_k$ are the spatial modes of the moments $\omega, \mu, \tau$ given by
\begin{align*}
    \omega_k(t,z) = \int h_k(v,t,z) \; \mathrm{d} v, ~~ \mu_k(t,z) = \int v h_k(v,t,z) \; \mathrm{d} v, ~~ \tau_k(t,z) = \int v^2 k_k(v,t,z) \; \mathrm{d} v.
\end{align*}
It can be shown that \eqref{BGKEquation} is equivalent to

\begin{align} \label{estimate25}
\frac{\partial}{\partial t} h_k + \mathrm{i} k \frac{2 \pi}{L} v h_k &= \sigma(z) \left(  g_0(v) \hat{h}_{k,0} + g_1(v) \hat{h}_{k,1} + g_2(v) \hat{h}_{k,2} - h_k \right), \; k \in \mathbb{Z}; \; t \geq 0.
\end{align}
For details of this derivation see \cite{Achleitner2018}. Since this derivation does not act on the $z$ variable, the derivation is exactly the same as in \cite{Achleitner2018}, so we will not repeat it here.
Now,  the vector of its Hermite coefficients satisfies

\begin{align*}
\frac{\partial}{\partial t} \hat{h}_k(t,z) + \mathrm{i} k \frac{2 \pi}{L} \mathbb{L}_1 \hat{h}_k(t,z) = - \sigma(z) \mathbb{L}_2 \hat{h}_k(t,z), \qquad k \in \mathbb{Z}; \; t \geq 0
\end{align*}

with the operators $\mathbb{L}_1, \; \mathbb{L}_2$ represented by the (infinite) matrices

\begin{align*}
&\mathbb{L}_1 := \begin{pmatrix}
0 & \sqrt{1} & 0 & \cdots \\
\sqrt{1} & 0 & \sqrt{2} & 0 \\
0 & \sqrt{2} & 0 & \sqrt{3} \\
\vdots & 0 & \sqrt{3} & \ddots \\
\end{pmatrix} , &
\mathbb{L}_2 := \text{diag} \;(0, \; 0, \; 0, \; 1, \; 1, \; \cdots)&
\end{align*}
Equivalently, we can also write
\begin{align} \label{BGKEquationFinal}
&\frac{\partial}{\partial t} \hat{h}_k(t,z) = - C_k \hat{h}_k(t,z) \quad k \in \mathbb{Z}; \; t \geq 0
\quad \text{with} \quad
 C_k := \mathrm{i} k \frac{2 \pi}{L} \mathbb{L}_1 + \sigma(z) \mathbb{L}_2.&
\end{align}
We note that this model satisfies the following conservation properties.

\begin{lemma}\label{lem:k=0}
The moments $\omega_0(t,z), \; \mu_0(t,z), \; \tau_0(t,z)$ satisfy $$\omega_0(t,z)=0, \quad \mu_0(t,z)=0, \quad\tau_0(t,z)=0,$$
for all $t>0$.
\end{lemma}
%\begin{proof}
This can be proven by multiplying \eqref{estimate25} for $k=0$ by $1, v, v^2$ and then integrating with respect to $v$. In the resulting equations one can compute the Maxwellian integrals and use the relations % \begin{align}
%\frac{\partial}{\partial t} \omega_0(t,z) &= \sigma (z) \Bigg( \int\limits_{-\infty}^{\infty} \mathrm{M}_1 (v) \; \mathbf{d} v \; \hat{h}_{0,0} + \int\limits_{-\infty}^{\infty} v \mathrm{M}_1 (v) \; \mathbf{d} v \; \hat{h}_{0,1} \nonumber \\
%&\quad + \frac{1}{\sqrt{2}} \bigg( \int\limits_{-\infty}^{\infty} v^2 \mathrm{M}_1 (v) - \mathrm{M}_1 (v) \; \mathbf{d} v \bigg) \hat{h}_{0,2} - \omega_0(t,z)\Bigg) \nonumber \\
%&= \sigma(z) \left( \hat{h}_{0,0} - \hat{h}_{0,0} \right) = 0, \label{estimate29} 
%\end{align}
%where we used
 \eqref{estimate26}, 
%\begin{align}
%\frac{\partial}{\partial t} \mu_0(t,z) &= \sigma (z) \Bigg( \int\limits_{-\infty}^{\infty} v \mathrm{M}_1 (v) \; \mathbf{d} v\; \hat{h}_{0,0} + \int\limits_{-\infty}^{\infty} v^2 \mathrm{M}_1 (v) \; \mathbf{d} v \; \hat{h}_{0,1} \nonumber \\
%&\quad + \frac{1}{\sqrt{2}} \bigg( \int\limits_{-\infty}^{\infty} v^3 \mathrm{M}_1 (v) - v \mathrm{M}_1 (v) \; \mathbf{d} v \bigg) \hat{h}_{0,2} - \mu_0(t,z) \Bigg) \nonumber \\
%&= \sigma(z) \left( \hat{h}_{0,1} - \hat{h}_{0,1} \right) = 0, \label{estimate30} 
%\end{align}
 \eqref{estimate27}.
%\begin{align}
%\frac{\partial}{\partial t} \tau_0(t,z) &= \sigma (z) \Bigg( \int\limits_{-\infty}^{\infty} v^2 \mathrm{M}_1 (v) \; \mathbf{d} v \; \hat{h}_{0,0} + \int\limits_{-\infty}^{\infty} v^3 \mathrm{M}_1 (v) \; \mathbf{d} v \; \hat{h}_{0,1} \nonumber \\
%&\quad + \frac{1}{\sqrt{2}} \bigg( \int\limits_{-\infty}^{\infty} v^4 \mathrm{M}_1 (v) - v^2 \mathrm{M}_1 (v) \; \mathbf{d} v \bigg) \hat{h}_{0,2} - \tau_0(t,z)\Bigg) \nonumber \\
%&= \sigma(z) \left( \hat{h}_{0,0} + \sqrt{2} \hat{h}_{0,2} - \hat{h}_{0,0} - \sqrt{2} \hat{h}_{0,2} \right) = 0 \label{estimate31}
%\end{align}
 \eqref{estimate28} to deduce that $\omega_0(t,z), \; \mu_0(t,z)$ and $\tau_0(t,z)$ are constant functions in $t$ and then equal to zero due to the assumption on the initial data \eqref{normalized}. %Next we expand \eqref{estunate35} in the x-Fourier series at $t=0$ to get

%\begin{align*}
%&\omega(x,0,z) := \sum_{k \in \mathbb{Z}} \omega_k (0,z) \; \mathbf{e}^{\mathbf{i} k \frac{2\pi}{L} x} ;&
%&\mu(x,0,z) := \sum_{k \in \mathbb{Z}} \mu_k (0,z) \; \mathbf{e}^{\mathbf{i} k \frac{2\pi}{L} x} ;& \\
%&\tau(x,0,z) :=\sum_{k \in \mathbb{Z}} \tau_k (0,z) \; \mathbf{e}^{\mathbf{i} k \frac{2\pi}{L} x}.
%\end{align*}

%This, together with \eqref{estimate32} shows that

%\begin{align*}
%&\omega_k (0,z) = \int\limits_{\tilde{T}} \omega(x,0,z) \; \mathbf{d} x = 0;&
%&\mu_k (0,z) = \int\limits_{\tilde{T}} \mu(x,0,z) \; \mathbf{d} x = 0 ;& \\
%&\tau_k (0,z) = \int\limits_{\tilde{T}} \tau(x,0,z) \; \mathbf{d} x = 0
%\end{align*}

%and because of \eqref{estimate29}, \eqref{estimate30} and \eqref{estimate31} this finishes the proof.
%\end{proof}

Since in the following, we also want to find estimates  for $\partial_z^{(n)}h,$ we will also consider the $n-th$ derivative of equation \eqref{BGKEquation} with respect to $z$, and get 
\begin{align} \label{BGKEquationDerivativeOrginal}
\partial _z^{(n)} \partial_t h(x,v,t,z) + v \partial _z^{(n)} \partial_x h(x,v,t,z) &= \partial _z^{(n)} \left(\sigma(z) \mathcal{L} \left( h (x,v,t,z) \right)  \right)
\end{align}
With the same approach as above, this leads to
\begin{multline} \label{BkGEquationDerivative}
\frac{\partial^{(n)}}{\partial z^{(n)}} \; \frac{\partial}{\partial t} \hat{h}_k(t,z) = \\ - \mathrm{i} k \frac{2 \pi}{L} \mathbb{L}_1 \frac{\partial^{(n)}}{\partial z^{(n)}} \hat{h}_k(t,z) - \sum_{i=0}^n \binom{n}{i} \frac{\partial^{(i)}}{\partial z^{(i)}} \sigma (z) \mathbb{L}_2 \frac{\partial^{(n-i)}}{\partial z^{(n-i)}} \hat{h}_k(t,z)
\end{multline} 
for $k \in \mathbb{Z}; \; t \geq 0$. Alternatively, directly differentiating \eqref{BGKEquationFinal} $n$ times with respect to $z$ leads to the same result.

\section{Decay rate for a linearized BGK model with uncertainties} \label{ChapterEstimate}
In this section, we will study the decay to equilibrium of the function $h$. For this, we will follow the strategy of \cite{Achleitner2018}. We define the matrices $P_k$ as
\begin{align} \label{P_k}
P_k := \begin{pmatrix}
\begin{matrix}
  1 & -\frac{\mathrm{i} \alpha}{k} & 0 & 0 \\
  \frac{\mathrm{i} \alpha}{k} & 1 & -\frac{\mathrm{i} \beta}{k} & 0 \\
  0 & \frac{\mathrm{i} \beta}{k} & 1 & -\frac{\mathrm{i} \gamma}{k} \\
  0 & 0 & \frac{\mathrm{i} \gamma}{k} & 1
  \end{matrix}
  & \vline & \bigzero \\
\hline
  \bigzero & \vline &
  \bigidentity
\end{pmatrix} \qquad k \in \mathbb{N}
\end{align}
with $I$ being the identity matrix and $\alpha, \; \beta, \; \gamma \; \in \mathbb{R}$ will be chosen later in an appropriate way. We start with the following lemma.
%\subsection{Basic inequality estimates}
\begin{lemma}
\label{th:MatrixEstimate}
Assume  $0 < \sigma _{min}  \leq \sigma (z) \leq \sigma _{max}$, $L> 0$. Choose the matrices $P_k$ as in \eqref{P_k} and $C_k$ from \eqref{BGKEquationFinal}. Then there exists an $\alpha _{max} > 0$, such that with $\alpha \in (0, \alpha _{max}), \beta = \sqrt{2} \alpha, \gamma = \sqrt{3} \alpha$ the matrices $P_k$ and $C_k^{\ast} P_k + P_k C_k$ are positive definite for all $k \in \mathbb{Z} \setminus \lbrace 0 \rbrace$ and
\begin{align*}
C_k^{\ast} P_k + P_k C_k \geq 2 \mu P_k %\qquad uniformly \; in \; \mid k\mid \in \mathbb{N}
\end{align*}
with a $\mu > 0$ independent of $k$.
\end{lemma}
\textcolor{black}{The proof consists of standard algebra derivations. Therefore, we will move the proof to the appendix.}
Because of the structure of $P_k$ we had to exclude the case $k=0$ in the proof above. We want to catch up this now. This case can be deduced from lemma \ref{lem:k=0}. If we insert this result into \eqref{BGKEquation}, we obtain for $k=0$
\begin{align} \label{CaseK=0}
\frac{\partial}{\partial t} h_0(v,t,z) = - \sigma(z) \; h_0(v,t,z).
\end{align}
Using Gronwall's lemma, this shows the decay in the case $k=0$. \\

\subsection{Decay estimate}
Now, we continue with the decay estimate on $h$. For this, we define 
\begin{align} \label{EntropyFunktional}
\mathcal{E}(h)(t,z) := \sum_{k \in \mathbb{Z}} \langle h_k(v,z), P_k h_k(v,z) \rangle_{L^2(\mathbb{M}_1^{-1})},
\end{align}
Here the matrices $P_0 := I$ and $P_k$ are regarded as bounded operators on $\ell^2(\mathbb{N}_0)$ (and thus also on $L^2(\mathbb{M}_1^{-1})$).
\begin{theorem}\label{th:decay estimate}
Let $h(t)$ be a solution of \eqref{BGKEquation} with $0  < L$, $0 < \sigma _{min}  \leq \sigma (z) \leq \sigma _{max}$ and $\mathcal{E} (h(0)) (z) < \infty$, then we have %with $\mathcal{E}$ being an entropy functional, then $\forall z \in \mathbb{O}$
\begin{align*}
\mathcal{E}\left( h(t) \right)(z) \leq \mathrm{e}^{- 2 \lambda t} \mathcal{E} \left( h(0) \right)(z)
\end{align*}
with some $\lambda > 0$ for all $z$.
\end{theorem}
\begin{proof}
%Let us define the entropy functional $\mathcal{E}(\tilde{f})$ by
%where $\tilde{f} := h(t)+\mathbb{M}_1$. 
Equation \eqref{CaseK=0} leads to
\begin{align*}
\frac{\partial}{\partial t} \left\langle h_0(v), P_0 h_0(v) \right\rangle_{L^2(\mathbb{M}_1^{-1})} &= \left\langle \frac{\partial}{\partial t} h_0(v), h_0(v) \right\rangle_{L^2(\mathbb{M}_1^{-1})} + \left\langle h_0(v), \frac{\partial}{\partial t} h_0(v) \right\rangle_{L^2(\mathbb{M}_1^{-1})} %\\
%&= - \left\langle \sigma(z) h_0(v), h_0(v) \right\rangle_{L^2(\mathbb{M}_1^{-1})} - \left\langle h_0(v), \sigma(z) h_0(v) \right\rangle_{L^2(\mathbb{M}_1^{-1})}
\\
&= - 2 \sigma(z) \left\langle h_0(v),h_0(v) \right\rangle_{L^2(\mathbb{M}_1^{-1})} \\
&\leq - 2 \sigma_{min} \left\langle h_0(v),h_0(v) \right\rangle_{L^2(\mathbb{M}_1^{-1})}
\end{align*}
and thus using lemma \ref{th:MatrixEstimate} we get
\begin{align*}
\frac{\partial}{\partial t} \mathcal{E}(h)(t,z) &:= \frac{\partial}{\partial t} \sum_{k \in \mathbb{Z}} \left\langle h_k(v,z), P_k h_k(v,z) \right\rangle_{L^2(\mathbb{M}_1^{-1})} \\
&= \sum_{k \in \mathbb{Z} \setminus \lbrace 0 \rbrace} \frac{\partial}{\partial t} \left\langle \hat{h}_k(z), P_k \hat{h}_k(z) \right\rangle_{\ell^2} +  \frac{\partial}{\partial t} \left\langle \hat{h}_0(z), P_0 \hat{h}_0(z) \right\rangle_{\ell^2}\\
&\leq - \sum_{k \in \mathbb{Z}\setminus \lbrace 0 \rbrace} \left\langle \hat{h}_k(z), (C_k^{\ast} P_k + P_k C_k) \hat{h}_k(z) \right\rangle_{\ell^2}  - 2 \sigma_{min} \left\langle \hat{h}_0(z), P_0 \hat{h}_0(z) \right\rangle_{\ell^2} \\
& \leq - 2 \mu \sum_{k \in \mathbb{Z} \setminus \lbrace 0 \rbrace} \left\langle \hat{h}_k(z), P_k \hat{h}_k(z) \right\rangle_{\ell^2} - 2 \sigma_{min} \left\langle \hat{h}_0(z), P_0 \hat{h}_0(z) \right\rangle_{\ell^2} \\
&= - 2 \mu \sum_{k \in \mathbb{Z} \setminus \lbrace 0 \rbrace} \left\langle h_k(v,z), P_k h_k(v,z) \right\rangle_{L^2(\mathbb{M}_1^{-1})} \\
& \quad  \; - 2 \sigma_{min} \left\langle h_0(v,z), P_0 h_0(v,z) \right\rangle_{L^2(\mathbb{M}_1^{-1})} \\
& \leq - 2  \lambda \; \mathcal{E}(h)(t,z)
\end{align*}
where we define $\lambda= \min \lbrace \mu, \sigma_{min} \rbrace$ with $\mu$ from \eqref{declayrate}. Applying Gronwall's lemma finishes the proof.
\end{proof}

%The decay rate $2 \lambda$ is explicitly computable because $\sigma_{min}$ is given and $\mu$ is computable as shown in theorem \ref{th:MatrixEstimate} and the remarks \ref{rm:explicit calculations}, \ref{rm:explicit calculations uniform error}. With the same reasoning one has to note that $2 \lambda$ is not the exact decay rate, but gives a reasonable lower bound.

\subsection{Decay estimates in z-derivatives}
For both, analytic and numeric reasons, one might also be interested in the decay of the $n$-th derivative of a solution with respect to the random variable $z$. %We define% In accordance with theorem \ref{th:decay estimate} we define the entropy functional as
For the following, we define
\begin{align*}
\mathcal{F} \left( f, g \right) := \sum_{k \in \mathbb{Z}} \left\langle f (k), P_k g(k) \right\rangle_{\ell^2} \; \; \text{for} \; \; f(k),g(k) :  \mathbb{Z} \mapsto \ell^2
\end{align*}
and we denote $|| f||_{\mathcal{F}}:= \sqrt{\mathcal{F}(f,f)}$.
%\end{definition}
%\begin{definition}[Entropy functional]\label{def:entropy functional}
%\begin{align*}
%\mathcal{E} \left( f, g \right) := \sum_{k \in \mathbb{Z}} \left\langle f (k), P_k g(k) \right\rangle_{\ell^2} \; \; with \; \; f(k),g(k) :  \mathbb{Z} \mapsto \ell^2
%\end{align*}
%\end{definition}

%as well as the set:

%\begin{definition}[The set $\wp$]\label{def:Set}
%\begin{align*}
%\wp := \left\lbrace f(k) : \mathbb{Z} \mapsto \ell^2 \; \mid \mathcal{E} \left( f, f \right) < \infty \right\rbrace.
%\end{align*}
%\end{definition}

%\begin{remark}
%Note that $\wp$ is a vector space (over $\mathbb{C}$) and $\mathcal{E} (\cdot, \cdot)$ is a scalar product on $\wp$ itself. We will denote its induced norm with $\Vert \cdot \Vert_{\mathcal{E}}$. These claims can be shown in an straight forward computation, which we want to skip here.
%\end{remark}

\subsubsection{Special case: $\sigma (z)$ linear in $z$}
 We will show that in the special case of linear random dependence, which means that $\sigma(z)$ is linear in $z$, the linearized BGK-equation \eqref{BGKEquation} still follows  an exponential decay with the same rate $\lambda$ as in the case without $z$ derivatives. %To keep the following proof clear, we identify the $n$-th derivative with respect to $z$ with $\bullet^{(n)} (z) :=  \frac{\partial^{(n)}}{\partial z^{(n)}} \bullet (z)$.

\begin{theorem}\label{th:decay in derivatives}
Let $h(t)$ be a solution of \eqref{BGKEquation} with $0  < L$, $0 < \sigma _{min}  \leq \sigma (z) \leq \sigma _{max}$. Further we assume $\sigma (z)$ to be linear in $z$ and $\mathcal{E}\left(\frac{\partial^{(n)}}{\partial z^{(n)}} h\right)(0,z) < \infty \; $ for all $ n \in \mathbb{N}_0$. Then, for all $n \in \mathbb{N}_0$ and for all $z$, we have
\begin{align} \label{claim1}
\sqrt{ \mathcal{E}\left(\frac{\partial^{(n)}}{\partial z^{(n)}} h\right)} (t,z)
\leq  \mathrm{e}^{- \lambda t} \sum_{i=0}^n \binom{n}{i} \left( \tilde{c} \; t \right)^i \sqrt{\mathcal{E}\left(\frac{\partial^{(n-i)}}{\partial z^{(n-i)}} h \right)} (0, z)
\end{align}
with the same positive $\lambda$ as in theorem \ref{th:decay estimate}.  Further if $\mathcal{E}\left(\frac{\partial^{(n)}}{\partial z^{(n)}} h\right)(0,z) \leq H^{2 n}$ for a constant $H>0$ and for all $n \in \mathbb{N}_0$ we can simplify \eqref{claim1} to
\begin{align} \label{claim2}
\sqrt{ \mathcal{E}\left(\frac{\partial^{(n)}}{\partial z^{(n)}}h\right)} (t,z)
\leq  \mathrm{e}^{- \lambda t} \left( H + \tilde{c} t \right)^n.
\end{align}
\end{theorem}
%\todo{Kannst du mir bestätigen, dass du mit $H$ eine Konstante gemeint hast?}
\begin{proof}
We want to show the claim in two steps. First, we prove that the inequality 
\begin{align} \label{estimate8}
\frac{\partial}{\partial t} ||\hat{h}_k^{(n)} (t,z) ||_{\mathcal{F}} \leq - \lambda \; || \hat{h}_k^{(n)} (t,z) ||_{\mathcal{F}} + \tilde{c} \; n \; || \hat{h}_k^{(n-1)} (t,z) ||_{\mathcal{F}}
\end{align}
holds for all $ n \in \mathbb{N}_0$. Then, this will imply \eqref{claim1}
%\begin{align} \label{estimate9}
%\left\Vert \hat{h}^{(n)}_k \right\Vert_{\mathcal{E}} (t,z) \leq \mathrm{e}^{- \lambda t} \sum_{i=0}^n \frac{n !}{\left( n-i\right)! \; i!} \left( \tilde{c} \; t \right)^i \left\Vert \hat{h}^{(n-i)}_k \right\Vert_{\mathcal{E}} (0,z) %\quad \forall z \in \mathbb{O} \subseteq \mathbb{R}.
%\end{align}
for all $z$ as a direct consequence of lemma \ref{lem:Li1}. To start with this we first note that because of
\begin{align*}
\sigma^{(n)} (z) = 0 ~~ \text{for all} ~~ n > 1, ~~ \sigma^{(1)} (z) = c_1
\end{align*}
with $c_1$ being a constant, equation \eqref{BkGEquationDerivative} simplifies to
\begin{align*}
\frac{\partial}{\partial t} \hat{h}^{(n)}_k(t,z) &= - \mathrm{i} k \frac{2 \pi}{L} \mathbb{L}_1 \hat{h}^{(n)}_k(t,z) - \sigma(z) \mathbb{L}_2 \hat{h}^{(n)}_k(t,z) - n c_1 \mathbb{L}_2 \hat{h}^{(n-1)}_k(t,z) \\
&= - \left( C_k \hat{h}^{(n)}_k(t,z) + n c_1 \mathbb{L}_2 \hat{h}^{(n-1)}_k(t,z) \right) \quad k \in \mathbb{Z}; \; t \geq 0
\end{align*}
with $C_k$ from \eqref{BGKEquationFinal}. Thus, for each $k \in \mathbb{Z} \setminus \lbrace 0 \rbrace$ we have
\begin{align*}
&\frac{\partial}{\partial t} \left\langle \hat{h}^{(n)}_k(t,z), P_k  \hat{h}^{(n)}_k(t,z)\right\rangle_{\ell^2} \\
&= \left\langle  \frac{\partial}{\partial t} \hat{h}^{(n)}_k(t,z), P_k  \hat{h}^{(n)}_k(t,z) \right\rangle_{\ell^2} + \left\langle \hat{h}^{(n)}_k(t,z), P_k \frac{\partial}{\partial t} \hat{h}^{(n)}_k(t,z) \right\rangle_{\ell^2} \\
&= - \left\langle C_k \hat{h}^{(n)}_k(t,z) + n c_1 \mathbb{L}_2 \hat{h}^{(n-1)}_k(t,z), P_k  \hat{h}^{(n)}_k(t,z) \right\rangle_{\ell^2} \\
&\quad - \left\langle \hat{h}^{(n)}_k(t,z), P_k \left( C_k \hat{h}^{(n)}_k(t,z) + n c_1 \mathbb{L}_2 \hat{h}^{(n-1)}_k(t,z) \right) \right\rangle_{\ell^2} \\
%&= - \left\langle C_k \hat{h}^{(n)}_k(t,z), P_k  \hat{h}^{(n)}_k(t,z) \right\rangle_{\ell^2}  -  \left\langle n c_1 \mathbb{L}_2 \hat{h}^{(n-1)}_k(t,z), P_k  \hat{h}^{(n)}_k(t,z) \right\rangle_{\ell^2} \\
%&\quad - \left\langle \hat{h}^{(n)}_k(t,z), P_k C_k \hat{h}^{(n)}_k(t,z) \right\rangle_{\ell^2} - \left\langle \hat{h}^{(n)}_k(t,z), n c_1 P_k \mathbb{L}_2 \hat{h}^{(n-1)}_k(t,z) \right\rangle_{\ell^2} \\
&= - \left\langle \hat{h}^{(n)}_k(t,z), \left( C_k^{\ast} P_k + P_k C_k \right) \hat{h}^{(n)}_k(t,z) \right\rangle_{\ell^2} \\
&\quad + \left\langle - n c_1 \mathbb{L}_2 \hat{h}^{(n-1)}_k(t,z), P_k  \hat{h}^{(n)}_k(t,z) \right\rangle_{\ell^2} + \left\langle \hat{h}^{(n)}_k(t,z), - n c_1 P_k \mathbb{L}_2 \hat{h}^{(n-1)}_k(t,z) \right\rangle_{\ell^2}.
\end{align*}
Thus using theorem \ref{th:decay estimate} we get
\begin{align} \label{estimate11}
&\frac{\partial}{\partial t} \left\langle \hat{h}^{(n)}_k(t,z), P_k  \hat{h}^{(n)}_k(t,z) \right\rangle_{\ell^2} \leq 
- 2 \mu \left\langle \hat{h}^{(n)}_k(t,z), P_k \hat{h}^{(n)}_k(t,z) \right\rangle_{\ell^2}\\ 
&+ \left\langle - n c_1 \mathbb{L}_2 \hat{h}^{(n-1)}_k(t,z), P_k  \hat{h}^{(n)}_k(t,z) \right\rangle_{\ell^2}
+ \left\langle \hat{h}^{(n)}_k(t,z), - n c_1 P_k \mathbb{L}_2 \hat{h}^{(n-1)}_k(t,z) \right\rangle_{\ell^2}. \nonumber
\end{align}
Now we want to get an estimate of the form \eqref{estimate11} for the case $k=0$. Using \eqref{CaseK=0} we get
\begin{align*}
\frac{\partial}{\partial t} h_0^{(n)}(v,t,z) &= \frac{\partial^{(n)}}{\partial z^{(n)}} \big( -\sigma (z) h_0(v,t,z) \big) 
= \sum_{i = 0}^n \binom{n}{i}  -\sigma^{(i)} (z) \; h_0^{(n-i)}(v,t,z) \\
&= -\sigma (z) h_0^{(n)}(v,t,z) - n c_1 h_0^{(n-1)}(v,t,z)
\end{align*}
and thus with the same arguments as in the estimate $ k \neq 0$ above
\begin{align} \label{estimate7}
&\frac{\partial}{\partial t} \left\langle  h_0^{(n)}(v,z), P_0  h_0^{(n)}(v,z) \right\rangle_{L^2(\mathbb{M}_1^{-1})} 
= - 2 \sigma(z) \left\langle  h_0^{(n)}(v,z), h_0^{(n)}(v,z) \right\rangle_{L^2(\mathbb{M}_1^{-1})} \\
& +  \left\langle - n c_1 h_0^{(n-1)}(v,z), h_0^{(n)}(v,z) \right\rangle_{L^2(\mathbb{M}_1^{-1})} 
+ \left\langle  h_0^{(n)}(v,z), - n c_1 h_0^{(n-1)}(v,z) \right\rangle_{L^2(\mathbb{M}_1^{-1})} \nonumber 
\end{align}
%\sout{Next we want to show that $\mathbb{L}_2 \hat{h}_k^{(n-1)}(t,z) \in \wp$. \\ Obviously for every $x = (x_0, x_1, x_2, \cdots) \in \ell^2$ we have}
%\begin{align} \label{estimate4}
%\sout{\left\Vert \mathbb{L}_2 x \right\Vert_{P_k}^2} &\sout{= \left\langle \mathbb{L}_2 x, P_k \mathbb{L}_2 x \right\rangle_{\ell^2} 
%= \sum_{i = 3}^{\infty} \left\vert x_i \right\vert^2
%\leq \sum_{i = 0}^{\infty} \left\vert x_i \right\vert^2
%= \left\langle x, x \right\rangle_{\ell^2}}
%\end{align}
%\sout{as well as}
%\begin{align} \label{estimate5}
%\sout{\left\langle x, x \right\rangle_{\ell^2} \leq \underbrace{\frac{1}{1-\alpha\sqrt{3+\sqrt{6}}}}_{:= \tilde{C}^2 \; > \; 1} \left\langle x, P_k x \right\rangle_{\ell^2}}
%\end{align}
%\sout{for all $ k \in \mathbb{Z}$ as shown in \eqref{estimate3}. Combining \eqref{estimate4} and \eqref{estimate5} leads to}
%\begin{align} \label{estimate12}
%\sout{\left\Vert \mathbb{L}_2 x \right\Vert_{P_k}^2 \leq \left\langle x, x \right\rangle_{\ell^2} \leq \tilde{C}^2 \left\Vert x \right\Vert_{P_k}^2.}
%\end{align}
%\sout{Summing up \eqref{estimate12} with $ x = \mathbb{L}_2 \hat{h}^{(n-1)}_k(t,z)$ we have}
%\begin{align*}
%\sout{\mathcal{E}\left( \mathbb{L}_2 \hat{h}_k^{(n-1)}(t,z), \mathbb{L}_2 \hat{h}_k^{(n-1)}(t,z) \right) \leq \tilde{C}^2 \underbrace{\mathcal{E} \left( \hat{h}_k^{(n-1)}(t,z), \hat{h}_k^{(n-1)}(t,z) \right)}_{< \infty}}
%\end{align*}
%\sout{which means $\mathbb{L}_2 \hat{h}_k^{(n-1)}(t,z) \in \wp$. } Therefore, ...
%\todo{Das Durchgestrichene in einem Satz: One can show.... ohne Einführen von $\wp$.}
Now we set $\lambda := \min \lbrace \mu, \sigma_{min} \rbrace$ and remember that $\left\langle \cdot, \cdot \right\rangle_{L^2(\mathbb{M}_1^{-1})} = \left\langle \hat{\cdot}, P_0 \hat{\cdot} \right\rangle_{\ell^2}$ with $P_0 =I$. Thus combining \eqref{estimate11} with \eqref{estimate7} and summing up over all $k \in \mathbb{Z}$ leads to
\begin{align} \label{estimate13}
&\frac{\partial}{\partial t} \mathcal{F} \left(  \hat{h}_k^{(n)}(t,z),  \hat{h}_k^{(n)}(t,z) \right) \leq - 2 \lambda \; \mathcal{F} \left(  \hat{h}_k^{(n)}(t,z),  \hat{h}_k^{(n)}(t,z) \right)  \\
& + \mathcal{F} \left( \tilde{h}_k(t,z), P_k  \hat{h}_k^{(n)}(t,z) \right)
+ \mathcal{F} \left( \hat{h}_k^{(n)}(t,z), P_k \tilde{h}_k(t,z) \right) \nonumber,
\end{align}
%\todo{Why is the term on the r.h.s. finite? Is there some implicit induction argument used?}
where we defined
\begin{align*}
\tilde{h}_k (t,z) := \begin{cases}
- n c_1 \hat{h}_0^{(n-1)}(t,z) \qquad \; \textcolor{black}{\text{if}}\quad  \; k = 0 \\
- n c_1 \mathbb{L}_2 \hat{h}_k^{(n-1)}(t,z) \quad \textcolor{black}{\text{if}} \quad \; k \neq 0. \\
\end{cases}
\end{align*}
%\sout{Note that $\tilde{h}_k(t,z) \in \wp$ because $\mathbb{L}_2 \hat{h}_k^{(n-1)}(t,z) \in \wp$ as shown above.} 
More precise, the only difference between $\tilde{h}_k(t,z)$ and $\mathbb{L}_2 \hat{h}_k^{(n-1)}(t,z)$ is the first summand (this is the case $k = 0$). %This term however, was already included in estimate \eqref{estimate12}.
%\sout{Now, because of
%\begin{align*}
%\frac{\partial}{\partial t} \mathcal{E} \left( \hat{h}^{(n)}_k(t,z),  \hat{h}^{(n)}_k(t,z) \right) \; \in \mathbb{R} ; \quad \mathcal{E} \left(  \hat{h}^{(n)}_k(t,z),  \hat{h}^{(n)}_k(t,z) \right) \; \in \mathbb{R}
%\end{align*}
%we have
%\begin{align*}
%\mathcal{E} \left( \tilde{h}_k(t,z), P_k  \hat{h}^{(n)}_k(t,z) \right) + \mathcal{E} \left( \hat{h}^{(n)}_k(t,z), P_k \tilde{h}_k(t,z) \right)  \quad \in \mathbb{R},
%\end{align*}
%too.
So, continuing estimate \eqref{estimate13}:
\begin{align*}
&\frac{\partial}{\partial t} \mathcal{F} \left(  \hat{h}^{(n)}_k(t,z),  \hat{h}^{(n)}_k(t,z) \right) \leq - 2 \lambda \; \mathcal{F} \left(  \hat{h}^{(n)}_k(t,z),  \hat{h}^{(n)}_k(t,z) \right) \\
& \quad \;+ \bigg\vert \mathcal{F} \left( \tilde{h}_k(t,z), P_k  \hat{h}^{(n)}_k(t,z) \right)
+ \mathcal{F} \left( \hat{h}^{(n)}_k(t,z), P_k \tilde{h}_k(t,z) \right) \bigg\vert \\
&\leq - 2 \lambda \; \mathcal{F} \left(  \hat{h}^{(n)}_k(t,z),  \hat{h}^{(n)}_k(t,z) \right) \\
& \quad \; + \left\vert \mathcal{F} \left( \tilde{h}_k(t,z), P_k  \hat{h}^{(n)}_k(t,z) \right) \right\vert
+ \left\vert \mathcal{F} \left( \hat{h}^{(n)}_k(t,z), P_k \tilde{h}_k(t,z) \right) \right\vert.
\end{align*}
%Now let $ \parallel \cdot \parallel_{\mathcal{E}} $ be the Norm induced by $ \mathcal{E} (\cdot, \cdot) $.
Using the Cauchy–Schwartz inequality leads to
\begin{align} \label{ZwischenschrittCauchySwarz}
&\frac{\partial}{\partial t} \mathcal{F} \left(  \hat{h}^{(n)}_k(t,z),  \hat{h}^{(n)}_k(t,z) \right) \\ &\leq - 2 \lambda \; \mathcal{F} \left(  \hat{h}^{(n)}_k(t,z),  \hat{h}^{(n)}_k(t,z) \right) \nonumber \\ 
& \quad \; + || \tilde{h}_k(t,z)||_{\mathcal{F}} \; || \hat{h}^{(n)}_k(t,z)||_{\mathcal{F}}
+ ||\hat{h}^{(n)}_k(t,z) ||_{\mathcal{F}}\; || \tilde{h}_k(t,z))||_{\mathcal{F}} \nonumber
\end{align}
We have the following relation% between $\mathcal{E} \left( \tilde{h}_k (t,z), \tilde{h}_k (t,z) \right)$ and \\ $\mathcal{E} \left( \hat{h}^{(n-1)}_k(t,z), \hat{h}^{(n-1)}_k(t,z) \right)$:
\begin{align*}
&\mathcal{F} \left( \tilde{h}_k (t,z), \tilde{h}_k (t,z) \right) \\
&= \left( n c_1 \right)^2 \bigg( \left\langle \hat{h}^{(n-1)}_0(t,z), P_0 \hat{h}^{(n-1)}_0(t,z) \right\rangle_{\ell^2} \\
& \quad \; + \sum_{k \in \mathbb{Z} \setminus \lbrace 0 \rbrace} \left\langle \mathbb{L}_2 \hat{h}^{(n-1)}_k(t,z), P_k \mathbb{L}_2 \hat{h}^{(n-1)}_k(t,z) \right\rangle_{\ell^2} \bigg) \\
&\leq \left( n c_1 \tilde{C} \right)^2 \mathcal{F} \left( \hat{h}^{(n-1)}_k(t,z), \hat{h}^{(n-1)}_k(t,z) \right)
\end{align*}
In the last inequality, we used the definition of $\mathbb{L}_2$. 
Now taking the roots, define $\tilde{c} := \vert c_1 \vert \tilde{C}$ and inserting into \eqref{ZwischenschrittCauchySwarz} leads to
\begin{align*}
\frac{\partial}{\partial t} ||\hat{h}^{(n)}_k(t,z)||_{\mathcal{F}}^2 &\leq - 2 \lambda \; ||\hat{h}^{(n)}_k(t,z)||_{\mathcal{F}}^2 + 2 n \tilde{c} || \hat{h}^{(n-1)}_k(t,z) ||_{\mathcal{F}} \; || \hat{h}^{(n)}_k(t,z)||_{\mathcal{F}}
\end{align*}
Dividing by $2 || \hat{h}^{(n)}_k(t,z) ||_{\mathcal{F}}$ gives \eqref{estimate8}.

Now, we can deduce \eqref{claim1} as it is described in the beginning of the proof. Finally inserting $\sqrt{\mathcal{E}\left(\tilde{h}^{(n)}\right)}(0,z) \leq H^n$ for all $n \in \mathbb{N}_0$ in \eqref{claim1} and using the binomial theorem leads directly to \eqref{claim2}. This finishes the proof.
\end{proof}

\subsubsection{General case under the assumption $\left\vert \frac{1}{n!} \frac{\partial^{(n)}}{\partial z^{(n)}} \sigma (z) \right\vert < C$}

The assumption that $\sigma (z)$ is linear in $z$, is very restrictive, so that our next goal is to loosen this condition. Therefore, from now on, the $z$-dependence of $\sigma(z)$ can be arbitrary, as long as $ \left\vert \frac{1}{n!} \frac{\partial^{(n)}}{\partial z^{(n)}} \sigma (z) \right\vert < C $ for all $n \in \mathbb{N}_0$, where $C$ is a constant independent of $n$. \textcolor{black}{Actually, this is a very weak constraint. It does not require that all derivatives have to be bounded by the same constant, the bound can grow with $n!$.} Further, we want to simplify the notation and set
\begin{align*}
\hat{h}_k^{(n)}(t,z) &:= \frac{\partial^{(n)}}{\partial z^{(n)}} \hat{h}_k (t,z) &
\tilde{h}_k^{(n)}(t,z) &:= \frac{\hat{h}_k^{(n)}(t,z)}{n!} \\
\sigma^{(n)} (z) &:=   \frac{\partial^{(n)}}{\partial z^{(n)}} \sigma (z) & \eta_k^{(n)}(t,z) &:= \mathbf{e}^{\lambda t} ||\tilde{h}_k^{(n)}(t,z) ||_{\mathcal{F}}.
\end{align*}
Then the following theorem, with the same explicit computable $\lambda$ as in theorem \ref{th:decay in derivatives}, holds:
\begin{theorem}%[Decay in derivatives (more) general case]
\label{th:decay in derivatives general}
Let $h(t)$ be a solution of \eqref{BGKEquation} with $0 < L$, $0 < \sigma _{min}  \leq \sigma (z) \leq \sigma _{max}$.% and $\mathcal{E}$ being a entropy functional defined in theorem \ref{th:decay estimate}. 
Further we assume $ \left\vert \frac{1}{n!} \sigma^{(n)} (z) \right\vert < C $ as well as $\mathcal{E}\left(\frac{\partial^{(n)}}{\partial z^{(n)}}\tilde{f}\right)(0,z) \leq H^{2n} \;$ for all $n \in \mathbb{N}_0$ for the initial data, then%for all $n \in \mathbb{N}_0$ and independent of $z$
, we have
\begin{align*}
\sqrt{\mathcal{E} \left( \frac{\partial^{(n)}}{\partial z^{(n)}} h \right)}(t,z) \leq \mathbf{e}^{- \lambda t} H^n + n! (1+H)^{n+1} \min \left\lbrace \mathbf{e}^{- \lambda t} (1 + \hat{C} t)^n, \; \mathbf{e}^{(\hat{C} - \lambda) t} 2^{n-1} \right\rbrace 
\end{align*}
for all $n \in \mathbb{N}$ with the same positive $\lambda$ as in theorem \ref{th:decay estimate} and a positive constant $\hat{C}.$
\end{theorem}
%\todo{What is $\hat{C}$? Wissen wir was \"uber das Vorzeichen von $\hat{C}- \lambda$? }
\begin{proof}
Repeating the same arguments as presented in the proof of theorem \ref{th:decay in derivatives} leads to
\begin{align*}
\frac{\partial}{\partial t} || \hat{h}_k^{(n)}(t,z) ||_{\mathcal{F}}^2 \leq - 2 || \hat{h}_k^{(n)}(t,z)||_{\mathcal{F}}^2 + 2 \tilde{C} \sum_{i=1}^n || \binom{n}{i} \sigma^{(i)} \hat{h}_k^{(n-i)}||_{\mathcal{F}} ||\hat{h}_k^{(n)}(t,z)||_{\mathcal{F}}
\end{align*}
Now we will use arguments presented in \cite{Li2017}. They are presented in lemma \ref{lem:Li2}. We first prove that all requirements are satisfied to use it.  Therefore we first use $ \left\vert \frac{1}{n!} \sigma^{(n)} (z) \right\vert < C $ for all $n \in \mathbb{N}_0$ to estimate further:
\begin{align*}
\frac{\partial}{\partial t} || \hat{h}_k^{(n)}(t,z) ||_{\mathcal{F}}^2\leq - 2 \lambda || \hat{h}_k^{(n)}(t,z)||_{\mathcal{F}}^2 + 2 \tilde{C} C ||\hat{h}^{(n)}_k(t,z)||_{\mathcal{F}} \sum_{i=1}^n \frac{n!}{(n-i)!} ||\hat{h}_k^{(n-i)}(t,z)||_{\mathcal{F}}% \mathcal{F}( \hat{h}_k^{(i)}(t,z))
\end{align*}
\textcolor{black}{We denote $\hat{C}:= \tilde{C} C$, shift the index in the sum, and divide by}
%where we denote  and simplify 
%\begin{align*}
%    \sum_{i=1}^n \frac{n!}{(n-i)!} || \hat{h}_k^{(n-i)}(t,z)||_{\mathcal{F}} = \sum_{i=0}^{n-1} \frac{n!}{i!} ||\hat{h}^{(i)}_k(t,z)||_{\mathcal{F}} %\leq \sum_{i=0}^{n-1} n! \mathcal{E}(\hat{h}^{(i)}_k(t,z))
%\end{align*}
%Dividing by 
$(n!)^2$ on both sides, we have
\begin{align*}
\frac{\partial}{\partial t} || \tilde{h}_k^{(n)}(t,z)||_{\mathcal{F}}^2 \leq - 2 \lambda ||\tilde{h}_k^{(n)}(t,z)||_{\mathcal{F}}^2 + 2 \hat{C} || \tilde{h}_k^{(n)}(t,z)||_{\mathcal{F}} \sum_{i=0}^{n-1} || \tilde{h}_k^{(i)}(t,z)||_{\mathcal{F}}
\end{align*}
and dividing by $ 2 || \tilde{h}_k^{(n)}(t,z) ||_{\mathcal{F}}$ leads to
\begin{align} \label{estimate14}
\frac{\partial}{\partial t} || \tilde{h}_k^{(n)}(t,z)||_{\mathcal{F}}\leq - \lambda || \tilde{h}_k^{(n)}(t,z)||_{\mathcal{F}} + \hat{C} \sum_{i=0}^{n-1} || \tilde{h}_k^{(i)}(t,z)||_{\mathcal{F}}.
\end{align}
\textcolor{black}{Then,we obtain for $\eta_k^{(n)}$ } %Note that 
%\begin{align*}
%\frac{\partial}{\partial t} \eta_k^{(n)}(t,z) = \mathbf{e}^{\lambda t}\left( \frac{\partial}{\partial t} ||\tilde{h}_k^{(n)}(t,z)||_{\mathcal{F}} + \lambda ||\tilde{h}_k^{(n)}(t,z)||_{\mathcal{F}} \right).
%\end{align*}
%Thus, multiplying \eqref{estimate14} with $\mathbf{e}^{\lambda t}$ results in
\begin{align*}
\frac{\partial}{\partial t} \eta_k^{(n)}(t,z) \leq \hat{C} \sum_{i=0}^{n-1} \eta_k^{(i)}(t,z).
\end{align*}
Because of \textcolor{black}{$\left( \frac{\partial^{(n)}}{\partial z^{(n)}} h \right) (0,z) \leq H^{2n}$} we have $\eta_k^{(n)}(0,z) \leq \frac{H^n}{n!}$, so that we can use lemma \ref{lem:Li2} point-wise in $z$ to get 
\begin{align} \label{estimate24}
\eta_k^{(n)}(t,z) \leq \frac{H^n}{n!} + (1+H)^{n+1} \min \left\lbrace (1 + \hat{C} t)^n, \; \mathbf{e}^{\hat{C} t} 2^{n-1} \right\rbrace .
\end{align}
Now we multiply \eqref{estimate24} with $\mathbf{e}^{- \lambda t}$ to reach
\begin{align*}
|| \tilde{h}_k^{(n)}(t,z)||_{\mathcal{F}} \leq \mathbf{e}^{- \lambda t} \frac{H^n}{n!} + (1+H)^{n+1} \min \left\lbrace \mathbf{e}^{- \lambda t} (1 + \hat{C} t)^n, \; \mathbf{e}^{(\hat{C} - \lambda) t} 2^{n-1} \right\rbrace .
\end{align*}
Multiplying with $n!$ finishes the proof.
\end{proof} 

\section*{Appendix. }

\subsection{Inequalities and estimates from the literature}

The following two inequalities had first been introduced in \cite{Li2017}. Even so we use a slightly different notation in our article, the proofs can be taken from their article.
\begin{lemma}{}\label{lem:Li1}
Assume $ J = [ 0,\infty)$, $n \in \mathbb{N}_0$ and a sequence $f_{(l)} \in C^1(J,\mathbb{R}) \;$ for all $l \in \lbrace 0, \cdots, n \rbrace$. If further the system of inequalities
\begin{align} \label{estimate15}
\frac{\partial}{\partial t} f_{(l)} \leq - \lambda f_{(l)} + C l f_{(l-1)}, \quad \qquad l \in \lbrace 0, \cdots, n \rbrace
\end{align}
with constants $\lambda, \; C > 0$  holds, then
\begin{align} \label{estimate16}
f_{(n)}(t) \leq \mathbf{e}^{-\lambda t} \sum_{i=0}^n \binom{n}{i} \left( C t \right)^i f_{(n-i)}(0),
\end{align}
where we set $f_{(-1)}$ to zero.
\end{lemma}
\begin{lemma}{}\label{lem:Li2}
Assume $ J = [ 0,\infty)$, $n \in \mathbb{N}_0$ and a sequence $f_{(l)} \in C^1(J, \mathbb{R}_{+}) \;$  for all $l \in \lbrace 0, \cdots,n\rbrace $.%\footnote{Here with $\mathbb{R}_{+}$ we denote all non negative real numbers.} If further the inequalities
\begin{align} \label{estimate18}
\frac{\partial}{\partial t} \tilde{f}_{(l)}(t) &\leq C \sum_{k = 0}^{l-1} \tilde{f}_{(k)}(t) \\
\tilde{f}_{(l)}(0) &\leq \frac{H^l}{l!} \nonumber
\end{align}
with constants $ \lambda, \; C > 0$, $H \geq 0$ and $\tilde{f}_{(l)}(t) := \mathbf{e}^{\lambda t} f_{(l)}(t)$  hold for all $l \in \lbrace 0, \cdots,n\rbrace$, then
\begin{align} \label{estimate19}
\tilde{f}_{(n)} (t) \leq \frac{H^n}{n!} + (1+H)^{n+1} \sum_{k=1}^n \frac{(C t)^k}{k!(k-1)!} \frac{(n-1)!}{(n-k)!}
\end{align}
and \eqref{estimate19} can further be relaxed to
\begin{align} \label{estimate20}
\tilde{f}_{(n)} (t) \leq \frac{H^n}{n!} + (1+H)^{n+1} \min \left\lbrace (1 + C t)^n, \; \mathbf{e}^{C t} 2^{n-1} \right\rbrace .
\end{align}
\end{lemma}

\subsection{Proof of lemma \ref{th:MatrixEstimate}}

\begin{proof}[Proof of lemma \ref{th:MatrixEstimate}]
Note first that $C_k^{\ast} P_k + P_k C_k$ has the form of a block-diagonal-matrix
\begin{align*}
\begin{pmatrix}
D_{k,\alpha , \beta, \gamma, \sigma (z)} & 0 \\
0 & \tilde{I} \\
\end{pmatrix}
\end{align*} 
with $\tilde{I}$ being $2 \sigma(z)$ times the (infinite dimensional) identity matrix and
\begin{multline*}
D_{k,\alpha ,\beta, \gamma, \sigma (z)} := \\ \mbox{\small$\begin{pmatrix}
 2 l \alpha & 0 & l \left(\sqrt{2} \alpha-\beta\right) & 0 & 0 \\
 0 & 2 l \left(\sqrt{2} \beta-\alpha\right) & 0 & l \left(\sqrt{3} \beta-\sqrt{2} \gamma\right) & 0 \\
 l \left(\sqrt{2} \alpha-\beta\right) & 0 & 2 l \left( \sqrt{3} \gamma - \sqrt{2} \beta \right) & -\frac{i \gamma \sigma(z)}{k} & 2 l \gamma \\
 0 & l \left(\sqrt{3} \beta-\sqrt{2} \gamma\right) & \frac{i \gamma \sigma(z)}{k} & 2 \sigma(z)-2 l \sqrt{3} \gamma & 0 \\
 0 & 0 & 2 l \gamma & 0 & 2 \sigma(z) \\
\end{pmatrix}$}
\end{multline*}
where $l := \frac{2 \pi}{L}$. Because of $0<2\sigma_{min}<2 \sigma(z)$ the matrix $\tilde{I}$ is already positive definite, such that it only remains to show the positive definiteness of $D_{k,\alpha ,\beta, \gamma, \sigma (z)}$. However, instead of seeking $\alpha, \; \beta\; , \gamma$ such that the matrix $D_{k,\alpha ,\beta, \gamma, \sigma (z)}$ is positive definite for all $k \in \mathbb{Z} \setminus \lbrace 0 \rbrace$, we simplify the problem by setting $\beta = \sqrt{2} \alpha$ and $\gamma = \sqrt{3} \alpha$. Thus, we get
\begin{align*}
D_{k,\alpha , \sigma (z)} := \begin{pmatrix}
 2 l \alpha & 0 & 0 & 0 & 0 \\
 0 & 2 l \alpha & 0 & 0 & 0 \\
0 & 0 & 2 l \alpha & -\frac{i \sqrt{3} \alpha \sigma(z)}{k} & 2 \sqrt{3} l \alpha \\
 0 & 0 & \frac{i \sqrt{3} \alpha \sigma(z)}{k} & 2 \sigma(z)-6 l \alpha & 0 \\
 0 & 0 & 2 \sqrt{3} l \alpha & 0 & 2 \sigma(z) \\
\end{pmatrix}
\end{align*}
which \textcolor{black}{will be an easier structure to analyze.} However, \textcolor{black}{we note} that we have to pay for this with a reduction of the decay rate. %With Sylvester's criterion one can find
Now we will use Sylvester's criterion to find 
a sufficient condition for $\alpha$, such 
that the matrix $D_{k,\alpha , \sigma (z)}$ is positive definite for all $k \in \mathbb{Z} \setminus \lbrace 0 \rbrace$
Therefore we define $\delta_{j} \left( k,\alpha,\sigma(z) \right)$ as the determinant of the lower right $ j \times j$ submatrix of $D_{k,\alpha , \sigma (z)}$ with $ 1 \leq j\leq 5$ and search for assumptions on $\alpha$, which lead to $\delta_{j} \left( k,\alpha,\sigma(z) \right) > 0$ for all $ 1 \leq j\leq 5 $. Thus we get
\begin{align*}
    \delta_1\left( k,\alpha,\sigma(z) \right) &= 2 \sigma(z) \\
    \delta_2\left( k,\alpha,\sigma(z) \right) &= 4 \sigma(z) \; ( \sigma(z) - 3 l \alpha) 
\end{align*}
\begin{align} \label{d3kd31}
\delta_3\left( k,\alpha,\sigma(z) \right) &= \alpha \left( 72 l^3 \alpha^2 - \left( 48 l^2 \sigma(z) + \frac{6 \sigma(z)^3}{k^2} \right) \alpha + 8 l \sigma(z)^2 \right) \nonumber \\
&\geq \alpha \left( 72 l^3 \alpha^2 - \left( 48 l^2 \sigma(z) + 6 \sigma(z)^3 \right) \alpha + 8 l \sigma(z)^2 \right) \nonumber \\
&= \delta_3\left( 1,\alpha,\sigma(z) \right)
\end{align}
\begin{align*}
    \delta_4\left( k,\alpha,\sigma(z) \right) %&= 2 \alpha^2 l \left( 72 l^3 \alpha^2 - \left( 48 l^2 \sigma(z) + \frac{6 \sigma(z)^3}{k^2} \right) \alpha + 8 l \sigma(z)^2 \right)\\
&= 2 \alpha l \; \delta_3\left( k,\alpha,\sigma(z) \right) \\
\delta_5\left( k,\alpha,\sigma(z) \right) %&= 4 \alpha^3 l^2 \left( 72 l^3 \alpha^2 - \left( 48 l^2 \sigma(z) + \frac{6 \sigma(z)^3}{k^2} \right) \alpha + 8 l \sigma(z)^2 \right)\\
&= 4 \alpha^2 l^2 \; \delta_3\left( k,\alpha,\sigma(z) \right)
\end{align*}
The first determinant $\delta_1$ is positive because of the assumption $0 <  \sigma (z)$. The second determinant $\delta_2$ is positive if we have
\begin{align} \label{condition1}
\alpha < \frac{\sigma(z)}{3 l},
\end{align}
whereas $\delta_3, \delta_4$ and $\delta_5$ are positive if %\footnote{In order to ensure $6 \alpha^2 <1$, which we will need later in the proof, we already drop the possibility of "large" $\alpha$ at this point.}
\begin{align} \label{condition2}
0<\alpha<\frac{8 l^2 \sigma(z) + \sigma(z)^3 - \sqrt{16 l^2 \sigma(z)^4 + \sigma(z)^6}}{24 l^3}.
\end{align}
So, to make sure that $D_{k,\alpha , \sigma (z)}$ is positive definite, we need to choose an $\alpha$ such that \eqref{condition1} and \eqref{condition2} hold. However, because of
\begin{align} \label{Zwischenabschätzung}
\frac{8 l^2 \sigma(z) + \sigma(z)^3 - \sqrt{16 l^2 \sigma(z)^4 + \sigma(z)^6}}{24 l^3} \leq \frac{\sigma(z)}{3 l} 
\end{align}
it is sufficient to find an $\alpha$ such that \eqref{condition2} is fulfilled. Equation \eqref{Zwischenabschätzung} is true since we have $\sqrt{16 l^2 \sigma(z)^4 + \sigma(z)^6} > \sigma(z)^3$.
However, it still remains to show that \eqref{condition2} can be fullfilled. 
So, we want to show that there exists an $\alpha_{max}$ such that
\begin{align} \label{condition3}
 0 < \alpha_{max} \leq \frac{8 l^2 \sigma(z) + \sigma(z)^3 - \sqrt{16 l^2 \sigma(z)^4 + \sigma(z)^6}}{24 l^3}.
\end{align}
Then \eqref{condition2} would be true for all $\alpha \in(0, \alpha_{max})$.
For proving this, we first note that
\begin{align} \label{estimate1}
0 &< \frac{1}{24 l^3} \; \sigma(z) \; \left( \sqrt{64 l^4 + 16 l^2 \sigma(z)^2 + \sigma(z)^4} - \sqrt{16 l^2 \sigma(z)^2 + \sigma(z)^4} \right) \nonumber \\
&=  \frac{8 l^2 \sigma(z) + \sigma(z)^3 - \sqrt{16 l^2 \sigma(z)^4 + \sigma(z)^6}}{24 l^3} := \alpha(l,\sigma(z))
\end{align}
 because of $\sigma(z) > 0$ and $l > 0$. Furthermore, $\alpha(l,\sigma(z))$ %\footnote{Remember that $\sigma(z)$ is a continuous function in $z$.} 
is a continuous function, such that if we take
\begin{align} \label{amax}
\alpha_{max} := \min_{\sigma(z) \in [\sigma_{min},\sigma_{max}]} \alpha(l,\sigma(z))
\end{align}
one gets $\alpha_{max} \leq \alpha(l,\sigma(z))$ for arbitrary fixed $l>0$. 
This, together with \eqref{estimate1} leads to \eqref{condition3}. 
Therefore, we get that $D_{k,\alpha , \sigma (z)}$ is positive definite for all $ \alpha \in ( 0, \alpha_{max})$ with $\alpha_{max}$ given by \eqref{amax}.

It remains to prove that the matrices $P_k$ are positive definite for this choice. In \cite{Achleitner2018} it is proven that $P_k$ is positive definite if $ \mid \alpha \mid^2 + \mid \beta \mid^2 + \mid \gamma \mid^2 < 1$. Since we set $\beta = \sqrt{2} \alpha$ and $\gamma = \sqrt{3} \alpha$ this reduces to $6 \alpha^2 < 1$.

One can compute that $\alpha(l, \sigma(z))$ takes its maximum at $l= \frac{\sqrt{3}}{4} \sigma(z)$ and we get
\begin{align} \label{maxaMax}
\alpha^2 \leq \alpha_{max}^2 < \left( \frac{8 l^2 \sigma(z) + \sigma(z)^3 - \sqrt{16 l^2 \sigma(z)^4 + \sigma(z)^6}}{24 l^3} \right)^2 < \frac{4}{9 \sqrt{3}}
\end{align}
if $l< \frac{\sigma_{max}}{2 \sqrt{2}},$
%at $\{\frac{\sqrt{3}\sigma(z)}{4}, \sigma(z)\}$. Further, we have $6 \alpha^2 \leq 6 \alpha_{max}^2 \leq  6 \left( \frac{4}{9 \sqrt{3}} \right)^2 = \frac{32}{81} < 1$ 
so that the matrices $P_k$ are also positive definite for all $ \alpha \in ( 0, \alpha_{max})$ with $\alpha_{max}$ from \eqref{amax} for $l< \frac{\sigma_{max}}{2 \sqrt{2}}$. %\todo{the condition seems to stringent... Is it also obtained with Sylvester criterion? Vielleicht nochmal nachrechnen?}

Next, we want to find a lower bound for the smallest eigenvalue of $C_k^{\ast} P_k + P_k C_k$. All eigenvalues of $\tilde{I}$ are $2 \sigma(z)$ and because of the block diagonal structure, $D_{k,\alpha , \sigma (z)}$ has a double eigenvalue $2 l \alpha$ together with the eigenvalues of its lower $ 3 \times 3$ submatix
\begin{align*}
D_{k,\alpha , \sigma (z)}^{(3)} := \begin{pmatrix}
2 l \alpha & -\frac{i \sqrt{3} \alpha \sigma(z)}{k} & 2 \sqrt{3} l \alpha \\
\frac{i \sqrt{3} \alpha \sigma(z)}{k} & 2 \sigma(z)-6 l \alpha & 0 \\
2 \sqrt{3} l \alpha & 0 & 2 \sigma(z) \\
\end{pmatrix}.
\end{align*}
Let $ \lbrace \lambda_1, \; \lambda_2, \; \lambda_3 \rbrace$ be the eigenvalues of $D_{k,\alpha , \sigma (z)}^{(3)}$ arranged in increasing order. So our aim is to find the minimum of the eigenvalues $2 \sigma(z), 2 l \alpha$ and $\lambda_1$. 
%\begin{align*}
%\sum_{i=1}^n \frac{x_i}{n} \geq \sqrt[n]{\prod_{i=1}^n x_i} \quad \text{for all non-negative} \quad \; x_i \in \mathbb{R}
%\end{align*}
%together with
%\begin{align*}
%&\sum_{i=1}^{3} \lambda_i = \mathrm{Tr} \; D_{k,\alpha , \sigma (z)}^{(3)},
%&\prod_{i=1}^3 \lambda_i = \mathrm{det} \; D_{k,\alpha , \sigma (z)}^{(3)}
%\end{align*}
We can estimate $\lambda_1$ from below by using the inequality of the arithmetic-geometric mean and get
%implies
\begin{align*}
\lambda_1(k,\alpha,\sigma(z)) &= \frac{\delta_3(k,\alpha,\sigma(z))}{\lambda_2 \lambda_3} \geq \delta_3(k,\alpha,\sigma(z)) \; \left( \frac{\lambda_2 + \lambda_3}{2} \right)^{-2} \\
&\geq \delta_3(k,\alpha,\sigma(z)) \; \left( \frac{\mathrm{Tr} \; D_{k,\alpha , \sigma (z)}^{(3)}}{2} \right)^{-2} 
= \delta_3(k,\alpha,\sigma(z)) \; \frac{1}{4\left( \sigma(z)-\alpha l \right)^2} > 0.
\end{align*}
since $D^{(3)}_{k, \alpha, \sigma(z)}$ is positive definite for $\alpha \in (0, \alpha_{max}).$
So, all in all, we need to find a lower bound of $\min \lbrace 2 l \alpha, \; \frac{\delta_3(k,\alpha,\sigma(z))}{4\left( \sigma(z)-\alpha l \right)^2}, \; 2 \sigma(z) \rbrace$. However, with  $ \alpha \in (0,\alpha_{max}) $ the following holds:
\begin{align} \label{Anhang}
\min \lbrace 2 l \alpha, \; \frac{\delta_3(k,\alpha,\sigma(z))}{4\left( \sigma(z)-\alpha l \right)^2}, \; 2 \sigma(z) \rbrace &= \min \lbrace 2 l \alpha, \; \frac{\delta_3(k,\alpha,\sigma(z))}{4\left( \sigma(z)-\alpha l \right)^2} \rbrace \nonumber \\
&\geq \min \lbrace 2 l \alpha, \; \frac{\delta_3(1,\alpha,\sigma(z))}{4\left( \sigma(z)-\alpha l \right)^2} \rbrace \nonumber \\
&= \frac{\delta_3(1,\alpha,\sigma(z))}{4\left( \sigma(z)-\alpha l \right)^2} := \lambda(l,\alpha,\sigma(z)).
\end{align}
The first equality is true due to \eqref{condition1}, the inequality follows from \eqref{d3kd31}. For the last equality
%We verify $(\ast)$ and $(\ast \ast)$ in lemma \ref{lem:Anhangabschätzung} in the appendix. 
%\begin{proof}[Proof of the inequalities ($\ast$) and ($\ast \ast$) in \eqref{Anhang}]
%In order to show $(\ast)$ we have to verify
%\begin{align*}
%2 l \alpha \leq 2 \sigma(z)
%\end{align*}
%with $\alpha \in (0, \alpha_{max})$. Note the definition of $\alpha_{max}$ in \eqref{amax} and \eqref{estimate1}, so multiplying \eqref{Zwischenabschätzung} with $2 l$ on gets
%\begin{align} \label{estimateanhang}
%2 l \alpha \leq \frac{2}{3} \sigma(z) < 2 \sigma(z).
%\end{align}
%This is $(\ast)$. To prove ($\ast \ast$) 
we have to show
\begin{align*}
\frac{ \alpha \left( 72 l^3 \alpha^2 - \left( 48 l^2 \sigma(z) + 6 \sigma(z)^3 \right) \alpha + 8 l \sigma(z)^2 \right) }{ 4 \left( \sigma(z) - \alpha l \right)^2 } \leq 2 l  \alpha
\end{align*}
with $\alpha \in (0, \alpha_{max})$. One can compute that this is equivalent to% Therefore we divide both sides by $2 l \alpha$ to obtain
%\begin{align*}
%\frac{36 \alpha^2 l^3 - 24 \alpha l^2 \sigma(z) + 4 l \sigma(z)^2 - 3 \alpha \sigma(z)^3}{4 l \sigma(z)^2 - 8 l^2 \sigma(z) \alpha + 4 \alpha^2 l^3} \leq 1.
%\end{align*}
%
%This is true, because of
%
\begin{align*}
%&36 \alpha^2 l^3 - 24 \alpha l^2 \sigma + 4 l \sigma^2 - 3 \alpha \sigma^3 \leq 4 l \sigma^2 - 8 \alpha l^2 \sigma + 4 \alpha^2 l^3 \\
%\Leftrightarrow \quad & 32 \alpha^2 l^3 - 16 \alpha l^2 \sigma - 3 \alpha \sigma^3 \leq 0 \\
%\Leftrightarrow \quad & 
\alpha  \bigg( 16 l^2 \left( 2 l \alpha - \sigma \right) - 3 \sigma^3 \bigg) \leq 0,
\end{align*}
which is true due to \eqref{condition1}.
%Thus everything is proven.
%\end{proof}
To get an estimate independent of $\sigma(z)$, we define for fixed $l>0$ and $\alpha \in (0,\alpha_{max})$
\begin{align} \label{lambdaMin}
\lambda_{min}(l,\alpha) := \min_{\sigma(z) \in [\sigma_{min},\sigma_{max}]} \lambda(l,\alpha,\sigma(z)) > 0.
\end{align}
Then %, because of $\mathbb{W} \left( \sigma(z) \right) \subseteq [\sigma_{min},\sigma_{max}]$ one gets $\lambda_{min}(l,\alpha) \leq  \lambda(l,\alpha,\sigma(z))$.
%Finally we get, if $P_k$ is chosen with some $\alpha \in (0,\alpha_{max}), \; \beta = \sqrt{2} \alpha$ and $ \gamma = \sqrt{3} \alpha$ uniformly $\forall \mid k \mid \in \mathbb{N}$, then
we get
\begin{align} \label{estimate2}
C_k^{\ast} P_k + P_k C_k \geq \lambda_{min}(l,\alpha) \; I.
\end{align}
Furthermore, a straight forward computation shows that the eigenvalues of $P_k$ are $ \lbrace 1 , \; 1 \pm \frac{ \alpha \sqrt{3+\sqrt{6}}}{k} , \;  1 \pm \frac{ \alpha \sqrt{3-\sqrt{6}}}{k} \rbrace$. These eigenvalues are positive for all $\; \alpha \in (0,\alpha_{max}), \; L>0, \; k \in \mathbb{N}$ according to \eqref{maxaMax}. Hence
\begin{align} \label{estimate3}
\left( 1 - \alpha \sqrt{3+\sqrt{6}} \right) \; I \leq P_k \leq \left( 1 + \alpha \sqrt{3+\sqrt{6}} \right) \; I
\end{align}
Combining \eqref{estimate2} and \eqref{estimate3} leads to
\begin{align} \label{declayrate}
C_k^{\ast} P_k + P_k C_k \geq 2 \mu  \; P_k
\end{align}
with $\mu = \frac{1}{2} \frac{\lambda_{min}(l,\alpha)}{\left( 1 + \alpha \sqrt{3+\sqrt{6}} \right)} >0,$
which completes the proof.
\end{proof}

\vskip2mm

%\par{\bf References.}\, %Always use $\backslash$cite$\{biblabelname\}$ (eg. \cite{taubes1}) to cite
      %    references which have been named in the bibliography via
       %   $\backslash$bibitem$\{biblabelname\}$.

          % Non-BibTeX users please use

          \end{document}